\documentclass{ws-jktr}

\usepackage{amssymb, amsgen, amscd, amssymb, amsmath, amsxtra}

\usepackage{graphics, epsfig, color}

\usepackage[matrix, arrow, curve]{xy}

\usepackage{amsmath,amscd}

\def\bee{\begin{equation*}}

\def\eee{\end{equation*}}

\def\be{\begin{equation}}

\def\ee{\end{equation}}

\def\C{{\mathbb C}}

\def\R{{\mathbb R}}

\def\Q{{\mathbb Q}}

\def\Z{{\mathbb Z}}

\definecolor{Maroon}{rgb}{0.8, 0.0, 0.0}

\begin{document}
\bibliographystyle{plain}

\markboth{David Auckly}
{TWO-FOLD BRANCHED COVERS}

\catchline{}{}{}{}{}

\title{TWO-FOLD BRANCHED COVERS}

\author{DAVID AUCKLY}

\address{Mathematics Department; Kansas State University; Manhattan, KS 66502}
\maketitle
\centerline{\tt{dav@math.ksu.edu}}



\begin{abstract}
Many three dimensional manifolds are two-fold branched covers of the three
dimensional sphere. However, there are some that are not. This paper includes
exposition about two-fold branched covers and includes many examples. It shows that there are
three dimensional homology spheres that do not two-fold branched cover any manifold,
ones that only two-fold branched cover the three dimensional sphere, ones that just
two-fold branched cover a non-trivial manifold, and ones that two-fold branched cover both
the sphere and non-trivial manifolds.
When a manifold is surgery on a knot, the possible quotients via involutions generically
correspond to quotients of the knot. There can, however, be a finite number of
surgeries for which there are exceptional additional symmetries. The included proof of
this result follows the proof of Thurston's Dehn surgery theorem. The paper also
includes examples of such exceptional symmetries. Since the quotients follow the
behavior of knots, a census of the behavior for knots with less than eleven crossings is
included.
\end{abstract}

\keywords{Branched covers; Hyperbolic Geometry; Dehn Surgery; Symmetries; Exceptional Surgeries}

\ccode{Mathematics Subject Classification 2000: 57M12, 57M25, 57M50, 57R65}

\section{Introduction}\label{intro}
Every closed, orientable $3$-manifold is a $3$-fold branched
cover of the
three dimensional sphere, $S^3$. This was independently shown by Hilden, Hirsh and
Montesinos \cite{hild,hir,mont}. Many three dimensional manifolds are two-fold branched covers of  $S^3$. However, there are $3$-manifolds
that are not. For example, R. H. Fox showed that the $3$-torus is
not a $2$-fold branched cover of $S^3$ \cite{fox}, and more generally Sakuma studied which fibered $3$-manifolds $2$-fold branched cover $S^3$, \cite{sa}. In the 1978 Kirby problem list Hilden and Montesinos asked if every homology sphere was a $2$-fold branched cover of $S^3$  \cite{ki}.
Earlier Montesinos proved that a rational homology sphere that $2$-fold branched covers $S^3$ may be obtained by rational surgery on a strongly invertible link, \cite{mont3}. But this left open the question of which manifolds take this form.
Myers answered the Kirby list question negatively by showing that the union of certain pairs of knot exteriors produced irreducible homology spheres that are not $2$-fold branched covers of $S^3$ or any other $3$-manifold \cite{my}. With Thurston's demonstration of the connection between geometric structures and the topology of $3$-manifolds, it became clear that geometric decompositions were the right way to address this and many other questions, \cite{thurston-notes}.

In this paper, we explain why many
$3$-manifolds are  $2$-fold branched covers of $S^3$, show that
there are hyperbolic integral homology spheres that are not $2$-fold branched
covers of $S^3$, but are $2$-fold branched covers of some manifold, and there are hyperbolic integral homology spheres that do not $2$-fold cover any $3$-manifold. The short explanation to the first point is the Montesinos trick relating rational tangle replacements in the branch locus in the base of a $2$-fold branched cover to Dehn surgery in the cover. The short explanation to the second is to study the geometric structure, or more generally the geometric decomposition of the total space, to identify the $2$-fold symmetries. When a $3$-manifold is obtained as surgery on a hyperbolic knot, the manifolds that are $2$-fold branched covered by it can be understood via the knot, and except for a finite number of possible exceptions there is a bijection between the branched covering projections of the knot complement and the branched covering projections of the manifold obtained by surgery. Understanding these exceptional symmetries is a problem similar to understanding the exceptional fillings of hyperbolic $3$-manifolds with just one cusp \cite{go}, but much less is known about the former question. We feel that this former question would be a good entry-level problem for a graduate student in topology, so we have included significant amounts of background in this exposition. More complete background on branched covers, Dehn surgery, Seifert surfaces, knot theory, and the fundamentals of low dimensional topology may be found in the book by Rolfsen \cite{rolfsen}.

I would like to thank B. Owens for asking me if there is a rational
homology sphere that is not a $2$-fold branched cover of $S^3$. I would also like to thank D. Kotschick, M. Sakuma, and the referee for their helpful comments on an earlier draft of this paper.
\vskip-.1in
\section{First Examples}
It is natural to consider $3$-manifolds that are obtained as surgery
on a knot. For knots with fewer than $11$ crossings, it turns out that most of these $2$-fold branch cover $S^3$. We will now discuss several qualitatively different examples of $2$-fold branched covers and review the standard constructions from low-dimensional topology that we use.

\vskip-.1in
\subsection*{Example 1}
A good example to consider first is $2/3$ Dehn surgery on the
$5_2$ knot displayed on the left in figure \ref{2/3slope}.
The knot and the framing curve are clearly set-wise invariant under a $2$-fold rotation of
$S^3$ about the axis given by the green line. This rotation will extend to the deck transformation of a
$2$-fold branched covering from the surgered manifold to $S^3$. Denote the rotation by $\tau:M\to M$, then the projection to the quotient $p:M\to M/\tau$ is a $2$-fold branched cover. It is locally modeled by the
map $p:\C\times\R \to \C\times\R$ given by $p(z,t)=(z^2,t)$. Given such a map, the deck transformations
are the self-maps of the domain $\tau$ so that $p\circ\tau = p$. For a 2-fold branched cover of a $3$-manifold, there is a unique non-trivial deck transformation that is an involution with (possibly empty) $1$-dimensional fixed point set, and vice-versa.

\begin{figure}[!ht]  
\hskip20bp
\includegraphics[width=50mm]{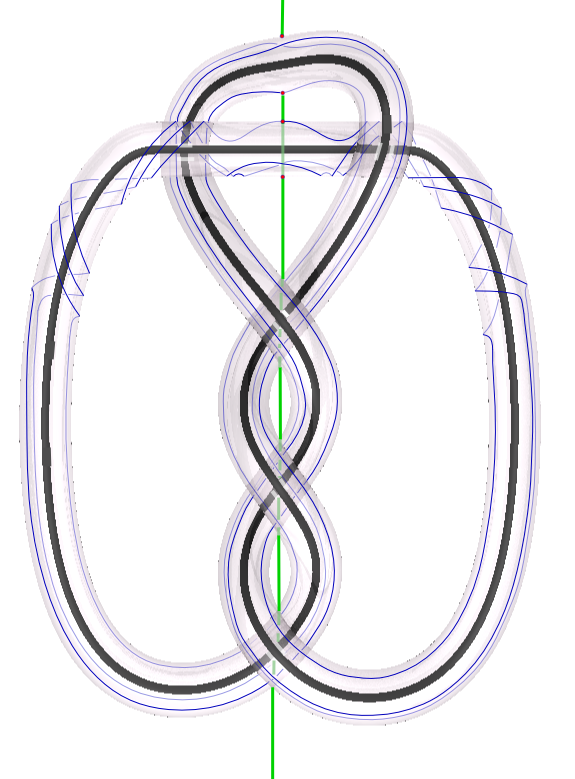}\hskip35bp
\includegraphics[width=50mm]{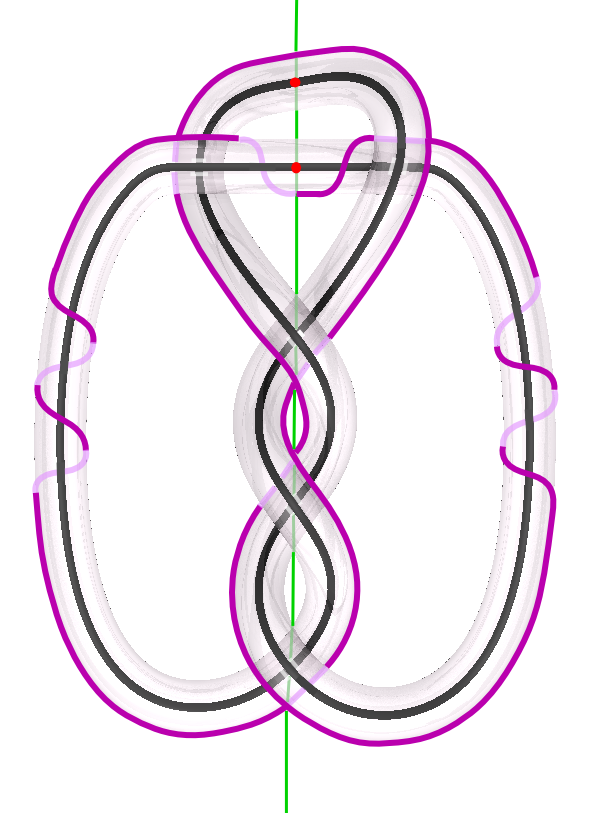}
\caption{The $2/3$ slope on the $5_2$ knot and the longitude}\label{2/3slope}
\end{figure}

Recall that Dehn surgery is the result of removing an open solid
torus and gluing in a closed solid torus:
\[
M_K(p/q):=\left(S^3-\stackrel{\circ}{N}\!(K)\right)\cup_{T^2}(D^2\times
S^1)\,.
\]
The open solid torus is a tubular neighborhood of a knot,
$\stackrel{\circ}{N}\!(K)$. The complement of an open tubular neighborhood
of a knot is known as the exterior of the knot. For oriented knots in $S^3$ generators
of the first homology of the boundary of the tubular neighborhood
can be chosen in a canonical way. The longitude, $\lambda$, is the
class of the oriented parallel to the knot that bounds in the
complement of the knot. The longitude of the $5_2$ knot is displayed on the right in figure \ref{2/3slope}. The meridian, $\mu$, is the class of $\{1\}\times \partial D^2$ when $S^1\times D^2$ is identified with the
tubular neighborhood of the knot so that its
linking number with the knot is positive.

The $2/3$-curve on the boundary of a
tubular neighborhood of the $5_2$ knot has a representative that is set-wise fixed by the
$2$-fold rotation. From the picture, it is not clear that every curve class has a representative that is set-wise fixed, and it is not clear what the quotient manifold is. We will see that every slope is represented by a set-wise fixed curve and the the quotient is just $S^3$ and
see that this is a very general phenomena.

We first verify that the original $2$-fold rotation induces an
involution on the surgered manifold. This rotation restricts to an
elliptic involution of the torus. This is the map $\tau:S^1\times S^1\to S^1\times S^1$ given
by $\tau(z,w):=(\bar z, \bar w)$ using complex coordinates. Here we identify the torus with the boundary of the exterior of the knot so that $(z,1)$ is a meridian and $(1,w)$ is a longitude.

In the surgery process we adjoin a solid torus $D^2\times S^1$ to the boundary of the exterior of the knot. Viewing $D^2\times S^1:=\{(z,w)\in\C^2 | |z|\le 1, |w|= 1\}$, the identification $j:S^1\times S^1\to D^2\times S^1$ can be specified by a matrix in $\text{SL}_2\Z$, say $A$, such that $j(\text{exp}(2\pi iv))=\text{exp}(2\pi iA^{-1}v)$ with $v=(x,y)$ and the exponential acting component wise. This is because every orientation preserving diffeomeorphism of a torus is isotopic to one of this form and changing the identification by an isotopy does not change the resulting manifold. These two facts are nicely explained in \cite{rolfsen}. The surgery coefficients are given by $p/q$ where $(p,q)=A(1,0)$.  Said another way, the identification  of $\partial(D^2\times S^1)$ and the boundary of $S^3-\stackrel{\circ}{N}\!(K)$ takes $(\partial D^2)\times \{1\}$ to the class $p\mu+q\lambda$.
Note that allowing orientation reversing maps would not produce any additional manifolds because there is an orientation reversing diffeomorphism of the torus that extends to a diffeomorphism of the solid torus.

To see that the involution extends to the surgered manifold, just define it on $D^2\times S^1$ by the same formula: $\tau(z,w):=(\bar z, \bar w)$ and notice that $j((\bar z, \bar w))=\overline{j(z,w)}$. Thus the original involution induces an involution on any manifold obtained by Dehn surgery on the knot. The fixed point locus in the knot exterior is a pair of intervals as is the fixed point set in the solid torus. Thus all of these manifolds are $2$-fold branched covers of some $3$-manifold.

\begin{figure}[!ht]  
\hskip55bp
\includegraphics[width=30mm]{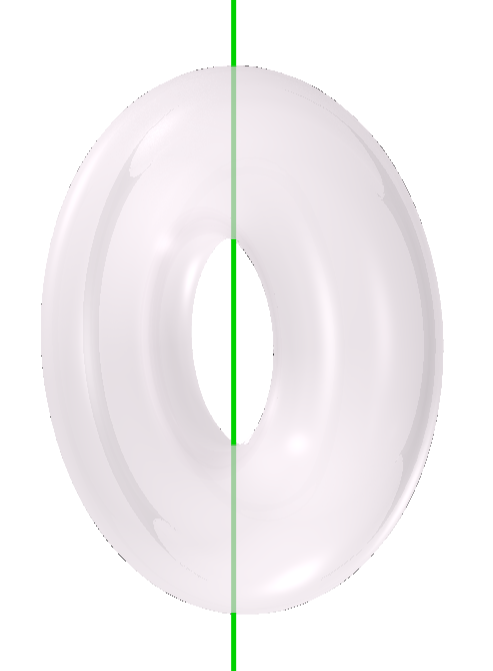}\hskip85bp
\includegraphics[width=30mm]{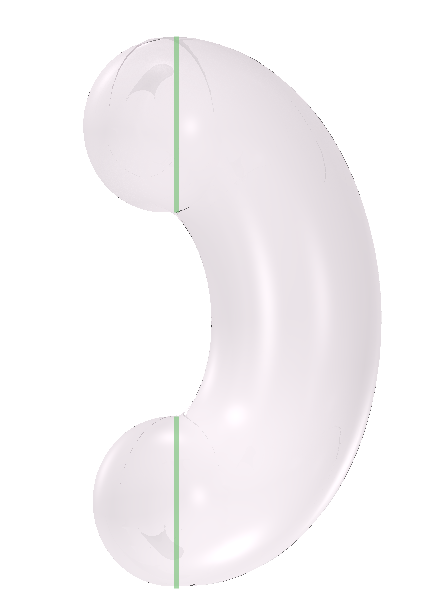}
\caption{Quotient of $D^2\times S^1$ by the involution}\label{fig2}
\end{figure}

An alternative proof that the involution extends will show that each of these manifolds is a $2$-fold branched cover of $S^3$. Figure \ref{fig2} shows that the quotient of a solid torus by the elliptic involution is a closed $3$-ball.
It follows that the quotient of Dehn surgery on a knot by an involution that acts as the elliptic involution on the boundary of the tubular neighborhood is obtained by removing a $3$-ball from the $3$-sphere (obtained as the quotient of the obvious $2$-fold involution on $S^3$ in the case of trivial $\infty:=1/0$ surgery on the knot)
and gluing in a closed $3$-ball via some homeomorphism. Since every homeomorphism of $S^2$ extends across the $3$-ball the quotient of the original manifold must be $S^3$.

\begin{figure}[!ht]  
\[\begin{array}{cc}
\includegraphics[width=30mm]{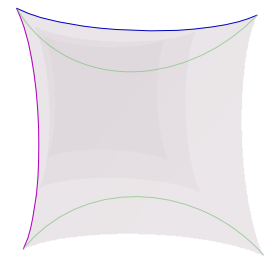} &\includegraphics[width=30mm]{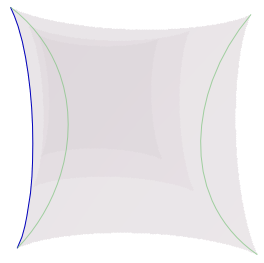}\\
\text{$\infty$ or $(1,0)$} & \text{$0$ or $S(1,0)$}\\
\includegraphics[width=30mm]{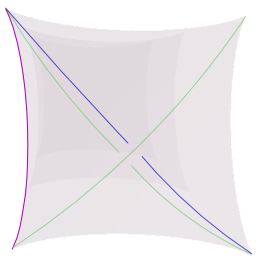} &\includegraphics[width=30mm]{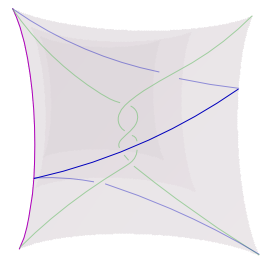}\\
\text{$1$ or $TS(1,0)$} & \text{$3$ or $T^3S(1,0)$} \\
\includegraphics[width=30mm]{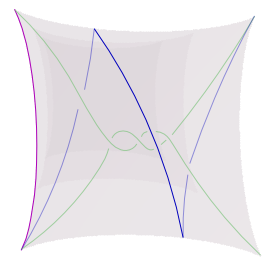} &\includegraphics[width=30mm]{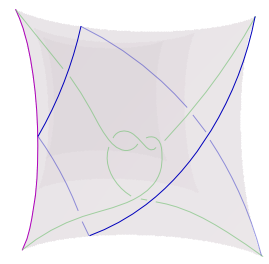}\\
\text{$-1/3$ or $ST^3S(1,0)$} & \text{$2/3$ or $TST^3S(1,0)$}
\end{array} \begin{array}{c} \\ \hskip5bp\includegraphics[width=55mm]{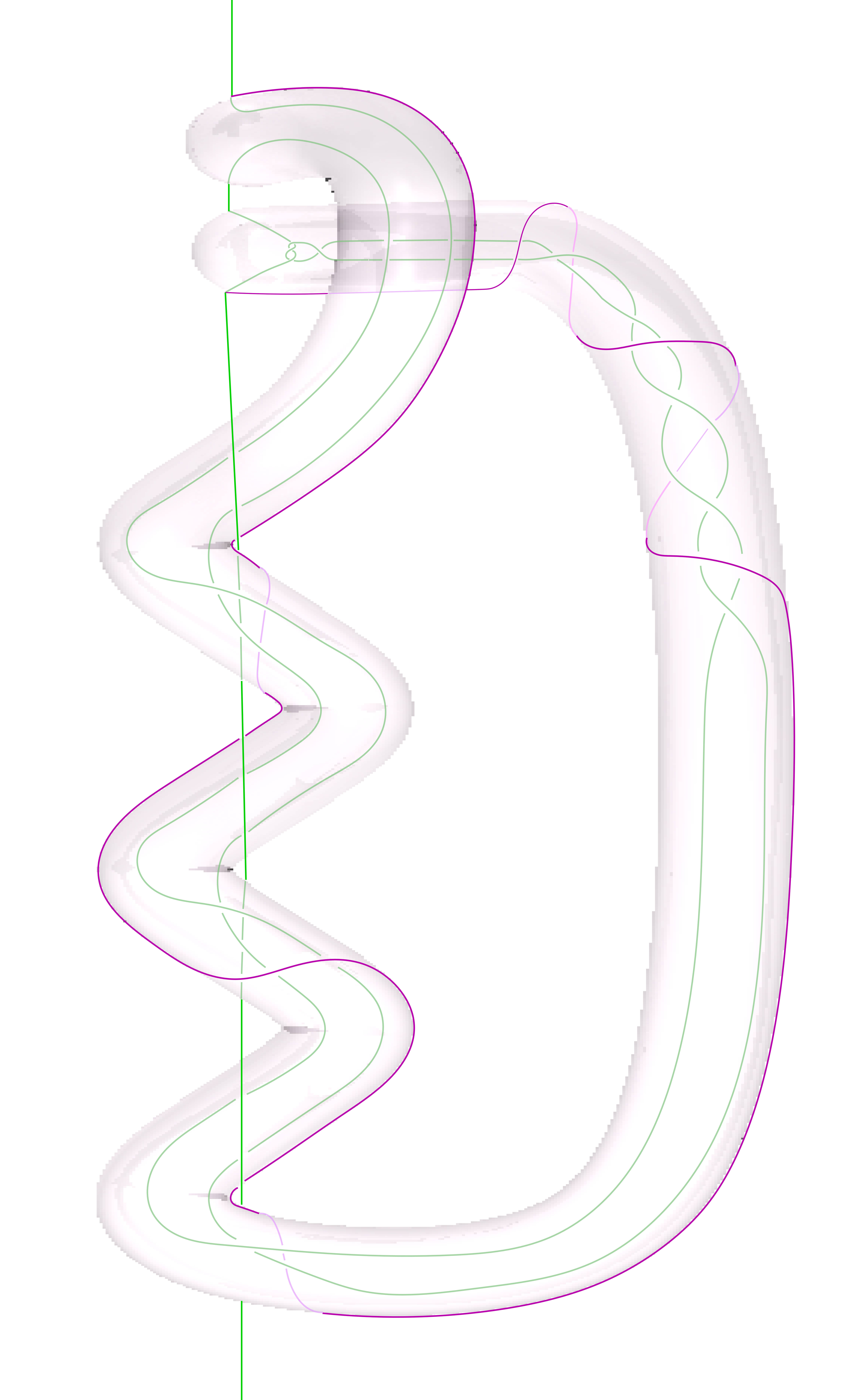} \\
\end{array}
\]
\caption{Tangling the branched set}\label{branchset}
\end{figure}

Every matrix in $\text{SL}_2\Z$ is a product of copies of $S:=\begin{bmatrix}0&-1\\1&0
\end{bmatrix}$ and \hfill\newline $T:=\begin{bmatrix}1&1\\0&1
\end{bmatrix}$. Since these both commute with $\begin{bmatrix}-1&0\\0&-1
\end{bmatrix}$, they also act on the quotient of $T^2$ by the elliptic involution. This quotient is an orbifold known as the pillow case or $2222$-orbifold. The underlying topological space is $S^2$ and the singular locus consists of four cone points of order $2$. The apparent left edge of the pillow case lifts to a longitude of the original knot. The apparent top edge of the pillow case lifts to a meridian.

The top left portion of figure \ref{branchset} displays the pillow case before any homeomorphism is applied. We labeled each pillow case in this figure by the corresponding surgery slope, and the action of the $S$ and $T$ matrices. The matrix $S$ acts as a $90$ degree rotation of the pillow case, and this clearly extends to the pillow. The matrix $T$ acts as a twist that interchanges the two corners on the right hand side. In each pillow the image of the top edge is colored blue, and the branch locus is colored   green. The apparent left edge is consistently  colored purple. If the resulting pillow is glued into the quotient of a knot exterior so that the apparent left edge maps to the image of the longitude and the apparent top edge maps to the meridian the resulting configuration will be the branch locus in the quotient manifold. Figure \ref{branchset} shows the result of gluing the $2/3$-pillow configuration into the image of the exterior of the $5_2$ knot. Adding the the blue curve to the quotient picture and taking a lift produces the invariant representative of the $2/3$-slope that was displayed in figure \ref{2/3slope}.

It should now be clear that this procedure could be followed with any surgery slope and any knot admitting an involution that restricts to the elliptic involution on the boundary of the tubular neighborhood. This procedure also produces an explicit description of the branch locus in the quotient. The rational tangle calculus and its relation to Dehn surgery is well known. It is called the Montesinos trick. A description of it applied to $1/q$ surgery may be found in \cite{fs}. This completes the alternate proof that the quotient of this type of involution is $S^3$.

\subsection*{Example 2}
\begin{figure}[!ht]  
\hskip55bp
\includegraphics[width=30mm]{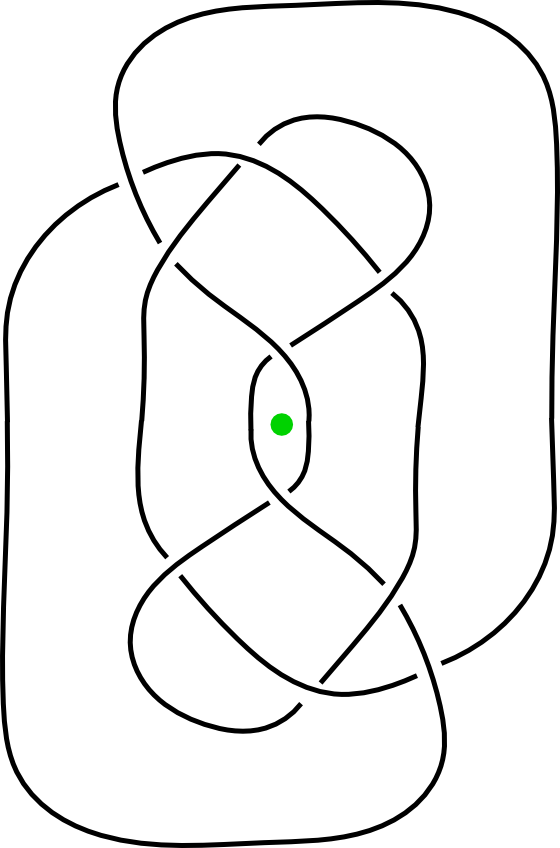}\hskip65bp
\includegraphics[width=30mm]{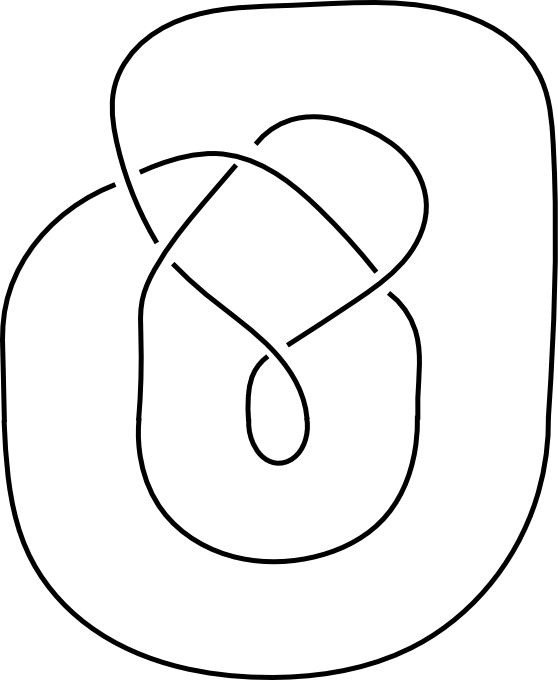}
\caption{The $10_{98}$ knot and its quotient}\label{10_98}
\end{figure}
Figure \ref{10_98} displays a projection of the $10_{98}$ knot. This projection is invariant under a $1/2$ rotation about the axis coming out of the page through the green dot. The main difference between this example and the previous example is that in this case the axis of the rotation does not intersect the knot. This means that the quotient of the pair $(S^3,K)$ is $(S^3,K^\prime)$ where $K^\prime$ is a copy of $S^1$ embedded in $S^3$. In this case it is knotted.

Once again we can verify that this rotation induces an involution on any manifold obtained by surgery on the $10_{98}$ knot. For any, but the trivial surgery, the quotient will be a manifold other than $S^3$. This follows from the Gordon Luecke theorem that only trivial surgery on a non-trivial knot produces $S^3$, \cite{GL}.

To begin, notice that the $2$-fold branched cover of a generic Seifert surface in the quotient is a Seifert surface for the original knot. This means that the $1/2$ rotation fixes a longitude of the knot set wise. The same rotation does not fix a meridian set wise. (It does fix it as a slope.) It follows that the boundary may be identified with a torus so that the rotation is given by $\tau(z,w)=(z,-w)$. For $p/q$ surgery on the knot, choose integers $r$ and $s$ so that $ps-qr=1$ and define a deck transformation on $D^2\times S^1$ by $\tau(z,w)=((-1)^rz,(-1)^pw)$. It is not difficult to check that this deck transformation is compatible with the identification of the boundary of $D^2\times S^1$ with the boundary of the tubular neighborhood of the knot.

A bit more information can be seen. When $p$ is even, $r$ must be odd and $\{0\}\times S^1$ will be an extra component in the branch locus since it is fixed set wise as one can see from the formula for $\tau$ above. Thus the branch locus will have two components, one in the exterior of the knot and one in the solid torus. When $p$ is odd, the axis in the knot exterior will be the only component of the branch locus. Since a meridian in the quotient lifts to two meridians, and a longitude in the quotient lifts to just one longitude, the surgery coefficient in the quotient will be exactly one half of the surgery coefficient on the original knot. Similar things can be said for knots with $n$-fold cyclic symmetry in which case $1/q$ surgery on the symmetric knot will $n$-fold branched cover $1/nq$ surgery on the quotient knot, \cite{fs}.

\section{Restrictions on covering projections}
We have seen two fairly general constructions of $2$-fold branched covers. In this short section we will explain how we can often identify all $2$-fold branched quotients of a manifold. The key idea is that a $2$-fold branched cover is determined by its unique non-trivial deck transformation and that in the hyperbolic case it suffices to consider isometries when looking for such deck transformations. State-of-the-art $3$-manifold theory is sufficiently advanced that this can be easily accomplished by quoting from the many deep theorems in the area. These include the orbifold theorem, the geometrization theorem and Thurston's hyperbolization theorem.

Intuitively, an orbifold is a space that is locally modeled on the quotient of Euclidean space by a finite group. For example, consider the quotient $\R^2$ by the group generated by a $1/3$rd rotation about the origin. Topologically, this is just a copy of $\R^2$. However, one should not treat all points the same. Every point other than the origin is moved by the non-trivial rotations. The origin is fixed by all rotations, so it makes sense to label the image of the origin with this stabilizer group $\Z_3$.   One needs to keep track of the stabilizer groups in the formal definition of an orbifold.

Another distinction between the orbifold and the underlying topological space arises when one considers compatible geometric structures. If one puts the standard metric on $\R^2$, the natural quotient metric will be a cone, so that the circumference of the circle of radius $1$ centered on the cone point is $2\pi/3$. Intuitively, a geometric orbifold will be locally equivalent to the quotient of a nice type of Riemannian manifold, known as a locally homogeneous space, by a finite group of isometries. An introduction to geometric orbifolds may be found in \cite{chk}.

Other $2$-dimensional examples, include the quotient of $S^2$ by a reflection through the equator, and the quotient of $S^2$ by the symmetry group of a tetrahedron. This last example should be viewed as $(\pi/2, \pi/3, \pi/3)$ spherical triangle, with one vertex labeled by the dihedral group of order four, two vertices labeled by the dihedral group of order $6$, and the edges labeled by $\Z_2$. The quotient of $S^2$ by the orientation preserving symmetries of a tetrahedron, is topologically a sphere, but it has a cone point labeled by $\Z_2$, and two cone points labeled by $\Z_3$. In these two-dimensional examples the underlying topological space is a manifold with boundary.

The underlying topological space need not be a manifold with boundary in general. For an example, the quotient of $\R^3/\{\pm 1\}$ has a singular point with non-simply connected link. Thus it is not a manifold.

To define the correct notion of a map between orbifolds is a bit tricky.
The modern definition of an orbifold is explained in \cite{adem}.
By this definition, an orbifold is essentially a special type of groupoid, and a groupoid is just a category such that every morphism has an inverse. Given a finite group acting on a manifold, one defines the objects of the groupoid to be the points in the manifold, and the arrows to be pairs $(x,g)$ consisting of a point in the manifold and an element of the group. To any object, i.e. point $x$, one may associate all arrows from the object to itself. This is just a fancy way to add groups to the points, namely the group of self morphisms of each point. In the case of a group acting on a manifold, the group associated to a point is just the stabilizer of the point. The topological space underlying an orbifold is just the set of equivalence classes of objects under the relation that identifies objects connected by an arrow. In the case of a group action, it is the quotient space. These are called {\it good} orbifolds. This definition allows one to describe examples that are not just the quotient of a manifold by a finite group (such orbifolds are called {\it bad}), as well as define maps between orbifolds.

In the case of a $2$-fold branched cover $p:M\to N$ we let the objects equal the points in $M$ and let the arrows be the pairs $(x,\tau)$ where $x\in M$ and $\tau$ is a deck transformation, i.e. a self homeomorphism of $M$ such that $p\circ \tau=p$.

The orbifold geometrization theorem, \cite{orbgeo,BLP} is stated below for the reader's convenience.
\vfill\newpage
\begin{theorem}[Thurston; Cooper, Hodgson, Kerkhoff; Boileau, Leeb, Porti]
If $\mathcal{N}$ is a good, compact, connected, orientable, irreducible, atoroidal, $3$-orbifold with non-empty singular locus, then $\mathcal{N}$ is geometric.
\end{theorem}

We will also use Thurston's hyperbolization theorem and the full geometrization theorem to establish the following result that provides effective restrictions on the possible quotients of $2$-fold covering projections.

\begin{theorem}\label{main}
If $M^3$ is a compact, connected, orientable, hyperbolic $3$-manifold and $p:M^3\to N^3$ is a $2$-fold branched cover, then there is a hyperbolic metric on $M$ for which every deck transformation is an isometry.
\end{theorem}

Note that up to isometry, any hyperbolic metric on such a manifold is unique, so there is an isometry between the new metric and the original metric.

\begin{proof}
We first notice that $N$ is the underlying space of an orbifold, $\mathcal{N}$.  We first consider the case when the branch set is non-empty and apply the orbifold theorem.
Clearly,  $\mathcal{N}$ is compact and connected. Let $\tau$ be the non-trivial deck transformation, and let $g$ be the metric on $M$. The averaged metric $\frac12(g+\tau^*g)$ is certainly $\tau$-invariant (but it might no longer be hyperbolic). If $\mathcal{N}$ was non-orientable, $d\tau$ would have to be an orientation reversing isometry
(with the averaged metric) of the tangent space of any fixed point, so $-I$ or  reflection in a plane. In either case, $N$ would fail to be a manifold near the orbit of the fixed point.

Recall, a bad orbifold is one that is not covered by a manifold. The quotient $\mathcal{N}$ cannot contain a bad orbifold because it is covered by a manifold. A spherical $2$-orbifold in the quotient would lift to a sphere in $M$. Since $M$ is hyperbolic, it is irreducible so the lift will bound a $3$-disk in $M$. The quotient of this $3$-disk will be a (possibly immersed) discal orbifold in $\mathcal{N}$ bounding the given spherical orbifold. It follows that $\mathcal{N}$ is irreducible. A similar argument (lift, find a compressing disc and take the quotient) will show that $\mathcal{N}$ is atoroidal.

Applying the orbifold theorem tells us that there is a geometric (locally homogeneous) metric on $\mathcal{N}$, say $h$. This lifts to a geometric metric $\tilde h$ on $M$ and this must be hyperbolic as the Gromov norm of a manifold with any of the other seven geometric structures must be zero and the Gromov norm of a hyperbolic manifold is proportional to its volume, \cite{ben, orbgeo}.

If the branch set is empty, one must still prove that the quotient is hyperbolic. If the quotient is orientable, we can apply Perelman's proof of the geometrization theorem to obtain the same result, \cite{per1,per2,per3}. This theorem is discussed
in \cite{lott,MT,morgan,bes}.  Some people may worry about the case when the quotient $N$ is non-orientable. The geometrization theorem is certainly true in this case because such an $N$ will be Haken, so the original work of Thurston applies, \cite{thurston}. See also the book \cite{kap}. (It is certainly appropriate to quote Thurston here as he is the one who originally pointed out the power of applying geometric techniques to the study of $3$-manifolds.)

In the orientable and non-orientable case the fact that the quotient is irreducible, atoroidal, with infinite fundamental group follows because the $2$-fold cover has the same properties. This suffices for the orientable case and the application of Perleman's theorem. To use Thurston's theorem in the non-orientable case, one must know that the quotient is Haken.
The fact that $N$ is Haken when it is non-orientable may be found in Hempel's book, \cite{H}.
\end{proof}

The previous theorem states that there is a hyperbolic metric on the manifold $M$ for which the deck transformations are isometires. One may worry that there are different hyperbolic metrics that need to be considered. However, it is known that there is a unique (up to isometry) hyperbolic metric on any closed hyperbolic $3$-manifold. This is Mostow's rigidity theorem, \cite{mostow}.

Given the previous theorem it makes sense to see what is known about the isometries of hyperbolic $3$-manifolds. Once again, the known results far surpass what is needed for our present problem. Kojima showed that every finite group could be realized as the group of isometries of a hyperbolic $3$-manifold \cite{K}. A corollary of this is that there is a hyperbolic $3$-manifold with trivial isometry group. Such a manifold
can not be a $2$-fold branched cover of any $3$-manifold. The construction does not
easily say anything about the topology of these manifolds.

One would like to say that any smooth automorphism of a manifold induces an isomorphism of the fundamental group. The problem with this is that such an automorphism may very well move the base point. If the manifold is connected one may connect the base point to its image with a path and thereby get an isomorphism of the fundamental group. The problem with this is that changing the choice of path will change the isomorphism by conjugation. The group of isomorphisms modulo conjugation (inner automorphism) is known as the outer automorphism group. It is denoted by $\text{Out}(\pi_1(M))$. Since homotopic maps induce the same map at the level of the fundamental group, there is a well defined map
\[
\text{Diff}(M)/\text{Diff}(M)_0 \to \text{Out}(\pi_1(M))\,.
\]
Here $\text{Diff}(M)$ is the group of all diffeomorphisms of $M$, and $\text{Diff}(M)_0$ denotes the path component of the identity. Of course any isometry is also a diffeomorphism, so one may map the isometries to the diffeomorphisms and the mapping class group $\text{Diff}(M)/\text{Diff}(M)_0$. It is a consequence of Mostow's rigidity theorem that the maps relating the three groups
$\text{Diff}(M)/\text{Diff}(M)_0$, $\text{Out}(\pi_1(M))$, and $\text{Isom}(M)$ are all isomorphisms for finite volume hyperbolic $3$-manifolds, \cite{ben,mostow}.

We remark that the following deep theorem due to D. Gabai gives even more information about the diffeomerphisms of a hyperbolic $3$-manifold.
\begin{theorem}[Gabai]
If ${M}$ is a closed, hyperbolic $3$-manifold, then the inclusion of the isometry group $\text{Isom}(M)$ into the diffeomorphism group $\text{Diff}(M)$ is a homotopy equivalence.
\end{theorem}

Hodgson and Weeks developed an algorithm that can compute the isometry groups of many hyperbolic $3$-manifolds, \cite{HW}. Basically, the isometries of the complement of any hyperbolic link, i.e. one with complement admitting a complete, finite volume hyperbolic metric, in a closed, hyperbolic $3$-manifold that take meridians to meridians induce isometries of the closed $3$-manifold. This gives a lower bound on the size of the isometry group. The total number of isometries of a closed, hyperbolic $3$-manifold is bounded above by the product of the order of the group of isometries fixing a given closed geodesic and the number of geodesics having the same invariant known as complex length. This algorithm is implemented in the SnapPy program.

One should not worry about the level of rigor of the computer calculations. Symmetries of knots can be analyzed with topological tools. The extensive theory of characteristic splittings of knots as developed by Bonahon and Siebenmann provides a powerful framework to address such questions \cite{BS}. Kodama and Sakuma used topological arguments to analyze the symmetry groups of all prime knots with fewer than eleven crossings, \cite{KS}.

As a first application of the fact that we can restrict our attention to isometries and it is possible to compute isometires, we computed the symmetry group of $S^3_{10_{98}}(1)$ to be $\mathbb{Z}_2$. Thus there is only one element of order two, and it is the involution generated by a $1/2$ rotation about the axis in figure \ref{10_98}. Thus it admits a double branched covering projection to $S^3_{3_{1}}(1/2)$, but does not double branch cover any other $3$-manifold (including $S^3$).

In addition to asking about the existence of $2$-fold branched cover quotients of a $3$-manifold, one may ask questions about the complexity of such quotients. One way to measure the complexity of such a quotient would be via the number of components of the branch set. The following result is
in Turaev's classification of oriented Montesinos links, \cite{turaev}  because the number of quasiorientations on the branch set is one less than the
$\text{\rm number of components}$, and the number of spin structures is
$\text{\rm dim}_{\Z_2}(H^2(M;\Z_2))$.

\begin{theorem}
If $\tau$ is an involution on a closed, connected, oriented $3$-manifold, $M$, and it acts trivially on the homology, then the
number of components of the fixed point set of $\tau$ satisfies the following bound:
\[
\text{\rm Number of components} = 1+\text{\rm dim}_{\Z_2}(H^2(M;\Z_2))\,.
\]
\end{theorem}

\section{More Examples}
In this section we consider additional examples representing the behavior of of all relevant involutions. There are exactly two involutions of a circle up to conjugation in the homeomorphism group: reflection, and half-rotation. Reflection fixes exactly two points, half-rotation has no fixed points. Similarly, up to conjugation by a homeomorphism there are exactly four involutions of $S^3$. They are given by sending $(x_1,x_2,x_3,x_4)$ to one of $(-x_1,x_2,x_3,x_4)$, or $(-x_1,-x_2,x_3,x_4)$, or $(-x_1,-x_2,-x_3,x_4)$, or $(-x_1,-x_2,-x_3,-x_4)$. This follows from the resolution of the Smith conjecture \cite{MB}. The involutions of $S^3$ are specified by their fixed point set. The fixed point sets of the four involutions are $S^2$, $S^1$, $S^0$, and $\emptyset$ respectively. Thus one can categorize knots that are setwise invariant under an involution into types according to the fixed point set in $S^3$ and the fixed point set in $S^1$. The possibilities
are:
\begin{itemlist}
\item[$(S^2,S^1)$] Here the only possibility is the unknot. This type of involution does not lead to branched covers.
\item[$(S^2,S^0)$] Here the only possibilities are the unknot and composite knots. This does not lead to branched covers either.
\item[$(S^1,S^1)$] Here the only possibility is the unknot. This gives the standard $2$-fold branched covering projection from $S^3$ to itself.
\item[$(S^1,S^0)$] This is fairly common among small crossing number knots. The $5_2$ knot admits such a symmetry as seen in our first example (figure \ref{2/3slope}). The knot projects to an arc in the quotient, thus this type of symmetry induces $2$-fold branched covering projection from any filling of such a knot complement to $S^3$.
\item[$(S^1,\emptyset)$] This is also fairly common among small crossing number knots. The $10_{98}$ knot admits such a symmetry (figure \ref{10_98}). This type of symmetry always leads to a $2$-fold branched covering projection. Here it is worth keeping track of whether the quotient of the knot is knotted or the unknot. This is because the quotient manifold will be a $3$-manifold other than $S^3$ when the original manifold is non-trivial and the quotient of the knot is knotted, or if the original manifold is not a homology sphere and the quotient of the knot is the unknot.
\item[$(S^0,S^0)$] This may occur as seen with the $8_{17}$ knot (figure \ref{810817}). The symmetry only extends to the trivial filling or to $0$-filling. Notice that the quotient of the trivial filling is not a $2$-fold branched cover because the quotient is not a manifold due to the singularity near the image of the fixed point. With $0$-filling the induced action has no fixed points and the resulting symmetry induces a $2$-fold branched cover ($2$-fold cover, in fact), to a non-orientable manifold.
\item[$(S^0,\emptyset)$] This may also occur, but it will never lead to a branched cover, as the quotient would have a neighborhood homeomorphic to a cone on the projective plane.
\item[$(\emptyset,\emptyset)$] This may also occur as seen with the $10_{155}$ knot (figure \ref{10155}). Since the $\Z_2$ action is free, the branch set would be contained in the solid torus added in the filling, or the induced quotient projection will be an unramified $2$-fold cover. The quotient will also be orientable, since the non-trivial deck transformation preserves orientation.
\end{itemlist}

\subsection*{Example 3}
For this paper, symmetries of types $(S^1,S^0)$ and $(S^1,\emptyset)$ will be the most important because these are the main types that produce interesting $2$-fold branched covers.
It is quite common for small crossing number knots to admit both types of symmetries. The $5_2$ knot admits both, as can be seen in the projection on the left of figure \ref{5_28_5}.
This projection is invariant under $1/2$ rotations about the blue horizontal axis, the green vertical axis, and an axis coming out of the center of the page. The blue axis misses the knot, and the other two axes both meet the knot in two points. The image of the knot in the quotient is just the unknot.

\begin{figure}[!ht]  
\hskip75bp
\[\begin{array}{cc}\begin{array}{c}
\includegraphics[width=35mm]{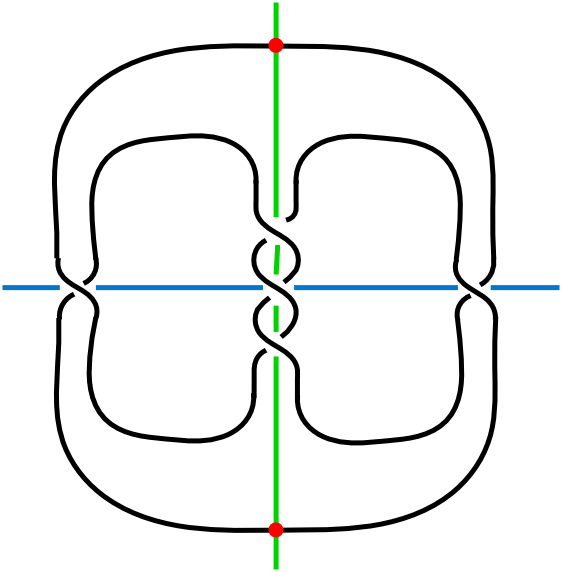}\hskip85bp \end{array}&
\begin{array}{c}\includegraphics[width=35mm]{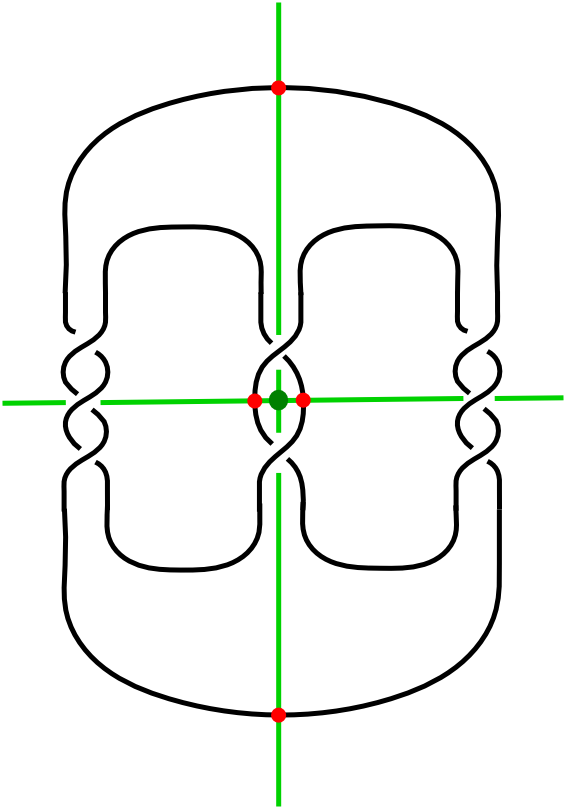}\end{array}
\end{array}
\]
\caption{Full symmetry of $5_2$ and $8_5$}\label{5_28_5}
\end{figure}

The $8_5$ knot has the same symmetry group, but the image of the knot under the quotient of the rotation that doesn't fix any point on the knot is knotted (it is a trefoil). This implies that any manifold obtained by non-trivial surgery on the $8_5$ knot is a $2$-fold branched cover of $S^3$ (via deck transformations induced by rotations that fix points on the knot), and is a $2$-fold branched cover over a $3$-manifold that is not $S^3$ (via deck transformations induced by rotations that do not fix any point on the knot.) This includes the integer homology spheres obtained by $1/n$ surgery. We know that non-trivial surgeries on non-trivial knots are non-trivial manifolds by the Gordon-Luecke theorem, \cite{GL}.

The same argument shows that any non-homology sphere that is surgery on the $5_2$ knot both $2$-fold branch covers $S^3$ and non-simply connected manifolds. The manifold $S^3_{5_{2}}(1/3)$ has symmetry $\mathbb{Z}_2\oplus\mathbb{Z}_2$ as expected from the left side of figure \ref{5_28_5}. Each of the involutions generates a double branched covering to $S^3$, thus this manifold only branched double covers $S^3$.

\subsection*{Examples 4 and 5}

We have seen that many knots admit symmetries of types $(S^1,S^0)$ and $(S^1,\emptyset)$, so that most surgeries on these knots $2$-fold branched cover $S^3$ and a non-simply connected manifold. The $10_{98}$ knot from example 2 had only a symmetry of type $(S^1,\emptyset)$, thus generic surgeries on it $2$-fold branched cover non-simply-connected manifolds, but do not $2$-fold branched cover $S^3$. In the other direction, the $8_{10}$ knot from figure \ref{810817} admits only a symmetry of type $(S^1,S^0)$. Thus generic surgeries on this knot $2$-fold branched cover $S^3$ but no other manifold.

\begin{figure}[!ht]  
\hskip55bp
\[\begin{array}{cc}\begin{array}{c}
\includegraphics[width=35mm]{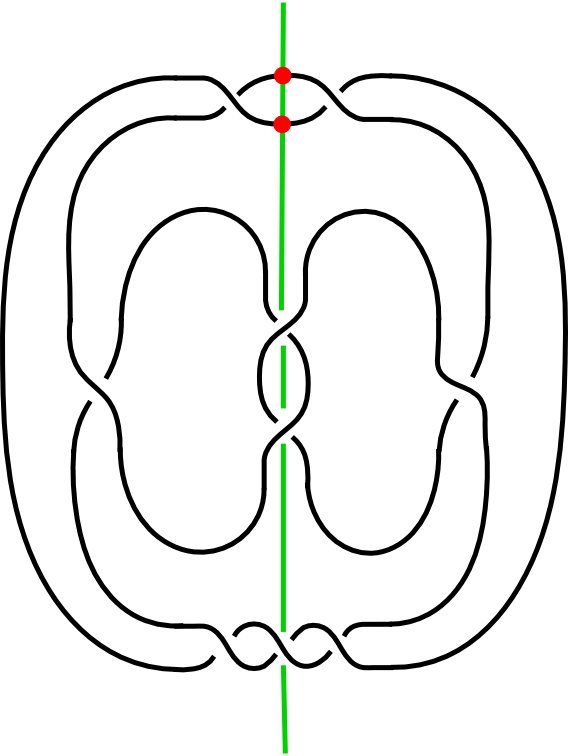} \end{array}&
\hskip20bp\begin{array}{c}\includegraphics[width=55mm]{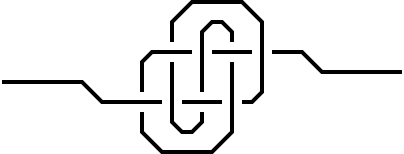}\end{array}
\end{array}
\]
\caption{The $8_{10}$ and $8_{17}$ knots}\label{810817}
\end{figure}

The $8_{17}$ knot from figure \ref{810817} is an interesting example. The only symmetry it admits is of type $(S^0,S^0)$. This symmetry restricts to a half rotation on the meridian and to a reflection on the longitude. For this symmetry to extend the surgery slope must be taken to itself. As $p\mu+q\lambda$ is mapped to $p\mu-q\lambda$ we see that this can only occur if $p/q=\infty$ or $p/q=0$. For zero surgery, the symmetry extends over the $S^1\times D^2$ via $(z,w)\mapsto (-z,\bar w)$. This has no fixed points, so the resulting quotient projection is a covering projection. Since the symmetry is orientation reversing, the quotient is non-orientable. Generic surgeries on this knot will have no symmetries, so the resulting manifolds will not $2$-fold branched cover any manifold.

The knot in figure \ref{92351111} has no symmetries, so generic surgeries on this knot do not $2$-fold branched cover any $3$-manifold. The smallest knot with no symmetries is the $9_{32}$ knot.

\subsection*{Example 6 and Torus Knots}

The quotient of $S^3$ by an involution with no fixed points is $\mathbb{R}P^3$. Thus to get a knot with a symmetry of type $(\emptyset,\emptyset)$, one may take any homologically non-trivial knot in $\mathbb{R}P^3$ and consider its lift into $S^3$. Easy examples are given by torus knots $T(p,q)$ with both $p$ and $q$ odd.

\begin{figure}[!ht]  
\hskip35bp
\includegraphics[width=105mm]{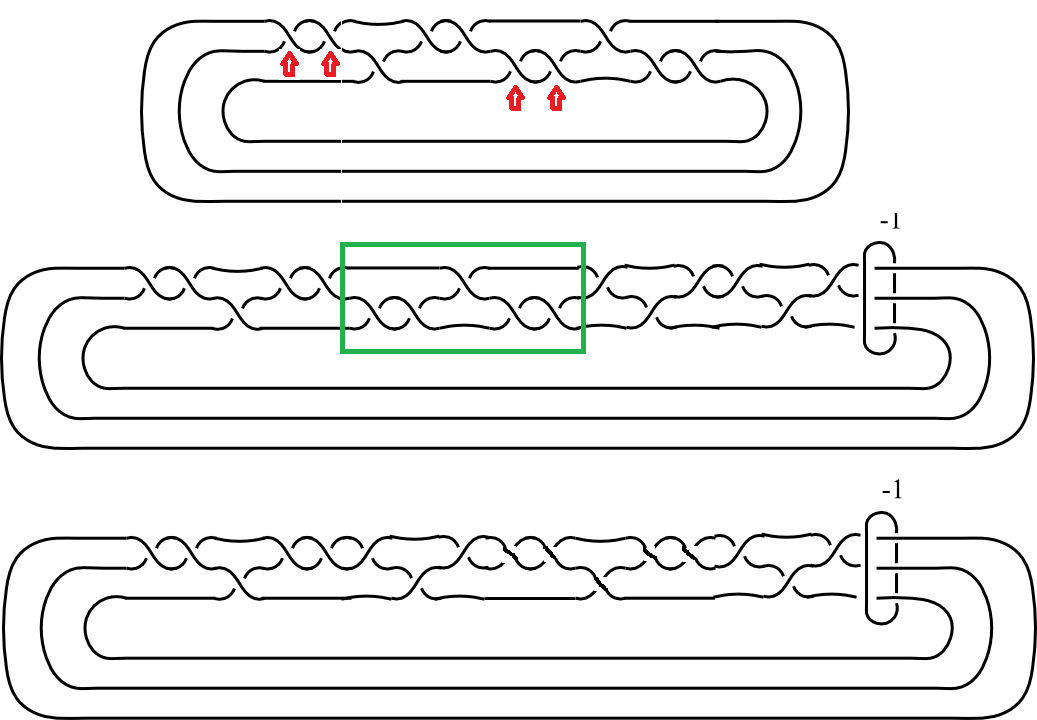}
\caption{The $T(3,5)$ ($10_{124}$) knot}\label{t35quot}
\end{figure}

Recall that a torus knot is given by $(\lambda^p,\lambda^q)$ inside $S^3$ viewed as pairs of complex numbers with $\text{max}(|z|,|w|)=1$. The antipodal involution $(z,w)\mapsto (-z,-w)$ clearly has no fixed points and preserves the torus knots $T(p,q)$ setwise for odd $p$ and $q$. It also takes the meridian to the meridian and the longitude to the longitude homologically, so any surgery extends. Since this involution has no fixed points, it does not even fix a meridian of the knot setwise. A longitude of the torus knot is given by the boundary of a Seifert surface. In this case a Seifert surface is given by the radial projection of $\{(z,w)\in\C^2 | z^q-w^p =1 \}$. It is clear that such a surface is disjoint from its image under the antipodal involution. It follows that this rotation does not fix a longitude of the knot setwise.

The special case of $T(3,5)$ is displayed in figure \ref{t35quot}, along with a surgery description. This figure is a bit different from the usual figure of this knot. The reason for using this presentation, is that it is closer to the $10_{155}$ knot that we will analyze next. In fact, changing the four crossings labeled with arrows will give the $10_{155}$ knot. The antipodal involution rotates both factors of the solid torus one-half way around. Indeed, rotating the right block of five crossings half way around a horizontal axis will make them match with the left block of five crossings, and rotating the block to the left will bring it to the position of the left block.

Notice that the knot fits fairly naturally in a solid torus, and
\[S^3=\partial (D^2\times D^2)=(\partial D^2)\times D^2 \cup D^2\times \partial D^2 =(S^1\times D^2) \cup (D^2\times S^1)\,.\]
Twisting one of the solid tori via a self-diffeomorphism $h_n: S^1\times D^2 \to S^1\times D^2$ given by $h_n(\lambda,z) = (\lambda, \lambda^n z)$, and changing the attaching map between the solid tori so the same boundary points of the solid tori are still identified will change the representation of the manifold. This process is called a {\it Rolfsen twist}, see \cite{rolfsen}.
A left-handed Rolfsen twist adds the full left twist that appears on the right side of the braid in the second part of figure \ref{t35quot}. Notice that the three crossings closest to the $-1$ framed circle just form a half twist about three strands, so the six crossing next to this circle are just one full left twist.

Giving the tangle in the green box a one-half rotation about the horizontal axis will have the effect of moving the closest left handed half twist to the other side of the box resulting in the representation on the bottom of the figure in which the antipodal involution may be seen as a one-half rotation. The quotient is then easily recognized as a homologically non-trivial knot in $\mathbb{R}P^3$. The same procedure will work with any knot having this type of symmetry. The quotient is displayed in figure \ref{t35quot2}. This figure further simplifies the quotient to make the Seifert fibered structure apparent.

\begin{figure}[!ht]  
\hskip35bp
\includegraphics[width=105mm]{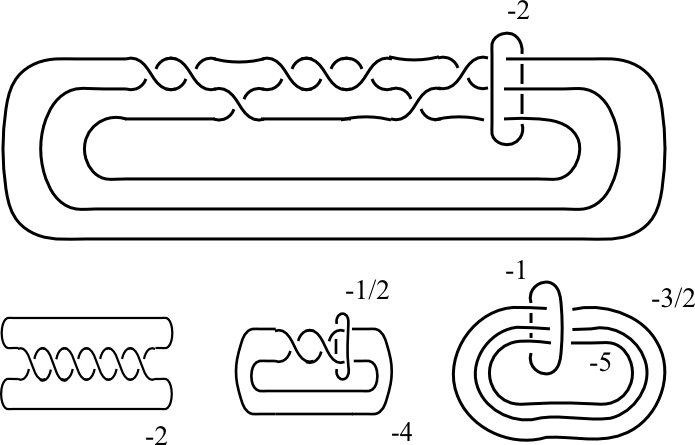}
\caption{The quotient of the $T(3,5)$ knot}\label{t35quot2}
\end{figure}

Since neither the meridian or longitude is fixed setwise, the rotation of the boundary torus induced by the antipodal involution is given by $(x,y)\mapsto(x+\frac12,y+\frac12)$. It follows that the indivisible homology class $r\mu+s\lambda$ on the boundary has a setwise fixed representative if and only if there is a $t\in\R$ such that $rt\cong st \cong 1 \ (\text{mod} 1)$, and this holds if and only if $r$ and $s$ are both odd.

When there is no such setwise fixed representative, the involution extends as a free involution to the $r/s$ filling, so the resulting filling is an unramified $2$-fold cover of a non-trivial $3$-manifold.
When there is a setwise fixed representative, the core of the filling torus will be the branch locus of the involution. The same arguments may be used with any knot admitting a symmetry of type $(\emptyset,\emptyset)$.

When  there is a setwise fixed representative, we do not need the explicit surgery descriptions to understand the quotients of the torus knots because the exterior of each of these has the structure of a Seifert fiber space given by the group action $(z,w)\cdot\lambda=(\lambda^pz,\lambda^qw)$.
The orbits of $(1,0)$ and $(0,1)$ are singular fibers of this Seifert fibration. The Seifert invariants of these singular fibers are $(p,u)$ and $(q,v)$ where $u$ and $v$ are integers with $pu+qv=1$.
For odd $(p,q)$ the antipodal involution preserves the fibers so the quotient will also be a Seifert fiber space. This remains the case for all Dehn fillings. The Seifert invariants of the quotient will be \[\{0,(Oo,0),(p,2u),(q,2v),(r,s)\}\,.\]
Unless $r=\pm1$ the fundamental group of this manifold will surject onto the $(p,q,r)$ triangle group. In general, the order of the first homology of this manifold is $|2r-pqs|$, and it is easy to check that this will never be one when $r=\pm1$ and $p$, and $q$ are odd and relatively prime.

While we are discussing torus knots, notice that the involution $(z,w)\mapsto (\bar z,\bar w)$ preserves any torus knot setwise, thus every one has a symmetry of type $(S^1,S^0)$ and therefore any surgery on a torus knot $2$-fold branched covers $S^3$. There is an important difference between torus knots and hyperbolic knots. Whereas the isometry group, mapping class group and outer automorphism group of a finite volume hyperbolic manifold are all isomorphic and all finite, this is no longer true for Seifert fiber spaces and torus knots. For example, the isometry group of the torus knots include the following $1$-parameter subgroup of isomorphisms: $f_\lambda(z,w)=(\lambda^pz,\lambda^qw)$, so the type $(\emptyset,\emptyset)$ isometries of torus knots described above are non-trivial as isometries, but are trivial elements of the mapping class group.

\begin{figure}[!ht]  
\hskip35bp
\includegraphics[width=105mm]{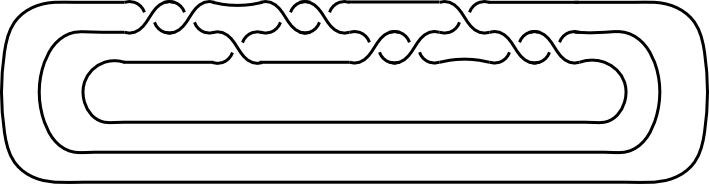}
\caption{The $10_{155}$ knot}\label{10155}
\end{figure}

The $10_{155}$ knot is displayed in figure \ref{10155}. It is obtained from the $T(3,5)$ torus knot by symmetrically changing four crossings. It follows that it has the same two involutions -- one of type $(S^1,S^0)$ and one of type $(\emptyset,\emptyset)$. We have analyzed several examples of $(S^1,S^0)$ involutions, as always the $(S^1,S^0)$ symmetry induces $2$-fold branched covering projections from any Dehn filling to $S^3$.

We will now describe the quotients arising from $(\emptyset,\emptyset)$ involutions in more detail. First, to construct a knot with this type of symmetry, one may start with a tangle, then glue it to a second copy rotated half way around a horizontal axis, and close the compound tangle. Assuming that the result is a knot, it will be a knot with an $(\emptyset,\emptyset)$ involution. The $T(3,5)$ torus knot has this structure, built from the tangle in the green box of figure \ref{t35quot}. The $10_{155}$ knot also has this structure. A method to construct a surgery description of the quotient of a Dehn filling with this type of symmetry was just described in the context of the $T(3,5)$ torus knot. The summary is to do a Rolfsen twist to  get a link with the structure of a satellite of one component of the Hopf link in which the pattern for the  has the form: rotated tangle, tangle, $1/2$ twist, $1/2$ twist. Rotating the unrotated tangle
gives rotated tangle, $1/2$ twist,  rotated tangle, $1/2$ twist. The $2$-fold symmetry is then apparent and the quotient is just a satellite of one component of the Hopf link with pattern: rotated tangle, $1/2$ twist.

It is not immediately clear if the quotient is $S^3$. When the original knot is hyperbolic, as in the case of the $10_{155}$ knot, the quotient of a filling will never be $S^3$.
Since the original knot is a hyperbolic knot, the quotient of its exterior under the free action will be hyperbolic. The question of which induced quotients will be $S^3$ is a special case of the exceptional surgery problem -- all but a finite number of surgeries are known to be hyperbolic by the work of Thurston.  The point in this case is that the quotient of a hyperbolic knot by a $(\emptyset,\emptyset)$ symmetry will be a hyperbolic (and therefore, non-trivial) knot in $\mathbb{R}P^3$, and a theorem of Kronheimer, Mrowka, Ozsv{\'a}th, and Szab{\'o}, proved using the same techniques that proved property P, states that a filling of such a knot can never be $S^3$, \cite{kmos}.

\section{Exceptional Symmetries}
We have seen that the possible quotients of $2$-fold branched covers of a $3$-manifold are determined by the involutions of the manifold. Thus if one understands all symmetries of the manifold one also understands all such quotients. When the manifold is surgery on a knot these symmetries can generally be understood via symmetries of the knot. This is similar to the situation with hyperbolic structures. Thurston's hyperbolic Dehn surgery theorem \cite{thds,thurston-notes} states that all but a finite number of Dehn fillings of a complete hyperbolic manifold with one cusp end admit hyperbolic structures. A similar thing holds true
for symmetries. For all but a finite number of Dehn fillings of a complete hyperbolic manifold with one cusp, any symmetry of the Dehn filling restricts to a symmetry of the original manifold. This is the content of our next result. In fact, the proof of this exceptional symmetry theorem follows the proof of the the hyperbolic Dehn surgery theorem. It relies on the fact that the length of the shortest geodesic tends to zero as the filling tends to infinity in Dehn surgery space, and for large fillings the shortest geodesic will be in the core of the filling. This observation has been used in the past. For example, Kojima used it when proving that two prime knots which have infinitely many homeomorphic branched covers are equivalent \cite{koj2}.

We describe families of hyperbolic structures, Dehn filling space, and the proof of Thurston's Dehn surgery theorem before addressing the exceptional symmetry theorem. Any complete hyperbolic $3$-manifold has the form $\mathbb{H}^3/\Gamma$ where $\Gamma$ is a discrete subgroup of isometries acting freely. This follows from something known as the Cartan-Hadamard theorem. Of course any space of this form is a complete hyperbolic $3$-manifold. It is clear that the fundamental group of $\mathbb{H}^3/\Gamma$ is $\Gamma$. The quotients by conjugate groups are homeomorphic and, in fact, isometric. Thus, to understand complete hyperbolic structures on a manifold $M$, one just needs to understand certain homomorphisms
$\pi_1(M) \to \text{Isom}(\mathbb{H}^3)$ up to conjugation.

The upper half space model for hyperbolic $3$-space is $\C\times (0,\infty)$ with metric $g=t^{-2}(dx^2+dy^2+dt^2)$ with $(x+{\bf i}y,t)$ the coordinates on $\C\times (0,\infty)$. One can check that geodesics are vertical lines, and open semicircles perpendicular to the boundary $\C\times \{0\}$. Since any point lies on at least $2$, and in fact, infinitely many, geodesics any isometry will be specified by the action of its natural extension to the boundary at infinity. It turns out that the orientation
preserving isometries act as linear fractional transformations, and any linear fractional transformation induces an orientation preserving isometry. As usual one may identifying linear fractional transformations and projective equivalence class of a matrices as:
 \[
 A(z)=\frac{az+b}{cz+d} \iff \begin{pmatrix} a&b\\c&d\end{pmatrix}\,.
 \]
Thus we identify $\text{Isom}(\mathbb{H}^3)$ with the projective special linear group, $\text{PSL}_2\C$, i.e. $2\times 2$ complex matrices with unit determinant mod out by non-trivial scale.

It is natural to look for homomorphisms of $\pi_1(T^2\times\R) = \langle\alpha,\beta\,|\, [\alpha,\beta]=1 \rangle$ up to conjugation. Notice that the set of homomorphisms of $\pi_1$ into $\text{SL}_2\C$ is described as the solution to a set of polynomial equations. (This is true for the homomorphisms of any group into $\text{SL}_2\C$.) This algebraic set may be given the subspace topology viewed as a subset of $\C^4$. The representation space may then be given the quotient topology. There are other ways to topologize this space, for example using geometric invariant theory. The arguments that we wish to make will work with several different topologies.  An alternative is to put a topology on the space of all subgroups of
$\text{PSL}_2\C$, then recognize that the image of any representation arising from a complete, finite volume hyperbolic metric is a closed subgroup.
More details about convergence of spaces of hyperbolic manifolds may be found in \cite{ben}.

The eigenvalues of any matrix in $\text{SL}_2\C$ must be $\lambda$ and $\lambda^{-1}$ since the determinant is $1$. If one eigenvalue of the matrix associated to $\alpha$ is not $1$, the matrix will be diagonalizable, so choosing an appropriate conjugate diagonal, and the only matrices that will commute with this will also be diagonal. These correspond to the linear fractional transformations of the form $z\mapsto \lambda z$. If $\lambda$ is a unit complex number the action will have fixed points $(0,t)$. In this case the isometry is called {\it elliptic}, but the quotient would not be a manifold. Thus, up to conjugacy a representation of $\pi_1(T^2)$ with one element going to a diagonal matrix will be given by $\rho(\alpha)=A$ and $\rho(\beta)=B$, with $A(z)=az$, $B(z)=bz$ where $|a|, |b|\neq 0,1$. Such isometries are called {\it hyperbolic}. One can check that these do not give rise to complete hyperbolic metrics on $T^2\times\R$. This means that the matrix of the image of $\alpha$ under a representation corresponding to a complete structure must have non-trivial Jordan form with a multiplicity $2$ eigenvalue of $1$ or $-1$. Via projective equivalence we may assume the eigenvalue is $1$. Up to conjugacy we may take $A(z)=z+1$ and $B(z)=z+\zeta$. Such isometries are called {\it parabolic}. For $\{1,\zeta\}$ independent over $\R$ it is clear that the quotient of this action on $\C$ is a torus, and that the quotient of this action on $\mathbb{H}^3$ is a complete hyperbolic structure on $T^2\times \R$.

When $A_w\to A_\infty$ and $B_w\to B_\infty$ in is natural to say the spaces $\mathbb{H}^3/\langle A_w, B_w\rangle$ converge to $\mathbb{H}^3/\langle A_\infty, B_\infty\rangle$. Set $A_\infty(z)=z+1$, $B_\infty(z)=z+\zeta$, $A_w(z)=e^{2\pi {\bf i}/w}z$, $B_w(z)=e^{2\pi \zeta {\bf i}/w}z$, and $E_w(z)=-\frac12 {\bf i} \left(\sin(2\pi/w)\right)^{-1}z+\left(\sin(2\pi/w)\right)^{-1}$. One then checks that $E_w\circ A_w\circ E_w^{-1}(z) = e^{4\pi {\bf i}/w}z + e^{2\pi {\bf i}/w}$ and
\[
E_w\circ B_w\circ E_w^{-1}(z) = e^{4\pi\zeta {\bf i}/w}z + e^{2\pi\zeta {\bf i}/w}\frac{(e^{2\pi\zeta {\bf i}/w}-e^{-2\pi\zeta {\bf i}/w})}{(e^{2\pi {\bf i}/w}-e^{2\pi {\bf i}/w})}\,.
\]
Finally, one sees that $E_w\circ A_w\circ E_w^{-1}\to A_\infty$ and $E_w\circ B_w\circ E_w^{-1}\to B_\infty$ as $w\to\infty$. Thus the spaces $\mathbb{H}^3/\langle A_w, B_w\rangle$ do converge to $\mathbb{H}^3/\langle A_\infty, B_\infty\rangle$ as $w\to\infty$.

The these spaces are parameterized by $w\in \overline{\C}=\C\cup\{\infty\}$ and this is the {\it hyperbolic Dehn filling space}. Setting $\rho_w:\pi_1(T^2)\to\text{PSL}_2\C$ to be the representation taking $\alpha$ to $A_w$ and $\beta$ to $B_w$, allows us to define a continuous map $w\mapsto \rho_w$ from Dehn filling space to the representation variety:
\[
DF: \overline{\C}\to\mathcal{R}(T^2)\,.
\]
To simplify our exposition of the relationship between Dehn surgery and Dehn filling space, we will now restrict to $\zeta = {\bf i}$. The arguments for general $\zeta$ are similar, with slightly more complicated algebra.
Let $p$ and $q$ be a pair of relatively prime integers and and consider the space corresponding to \newline $w=p+q{\bf i}$. One has $\rho_{p+{\bf i}q}(\alpha^p\beta^q)=A_{p+q{\bf i}}^pB_{p+q{\bf i}}^q=1$.
Pick integers $n$ and $m$ so that $pn-qm=1$, then $\alpha^p\beta^q$ and $\alpha^m\beta^n$ generate $\pi_1(T^2)$, so $\langle A_{p+q{\bf i}}, B_{p+q{\bf i}}\rangle$ is generated by $A_{p+q{\bf i}}^mB_{p+q{\bf i}}^n$. It follows that the fundamental group of $\mathbb{H}^3/\langle A_{p+q{\bf i}}, B_{p+q{\bf i}}\rangle$ is infinite cyclic and the space is homeomorphic to $S^1\times \R^2$. This can be constructed by removing $T^2\times (1,\infty)$ from $T^2\times (0,\infty)$ and glueing in $S^1\times D^2$ so that $\{1\}\times \partial D^2$ is glued to a curve in the homotopy class of $\alpha^p\beta^q$.

We can even understand the geometry of this space. Any point of the form $(0,t)$ is the maximum of each of a pair of transverse semicircles that are perpendicular to $\C\times \{0\}$. The map rotates the feet of these semicircles some amount and dilates them by $|e^{2\pi m i/(p+q{\bf i})}e^{-2\pi n /(p+q{\bf i})}|=e^{-2\pi/(p^2+q^2)}$. It follows that the point $(0,e^{2\pi/(p^2+q^2)})$ is equivalent to the point $(0,1)$.
Thus, the geodesic $\{0\}\times (0,\infty)$ projects to a geodesic in $S^1\times \R^2$ of length
\[
\int_1^{e^{2\pi/(p^2+q^2)}} t^{-1}\,dt=2\pi(p^2+q^2)^{-1}\,.
\]
This geodesic is the core geodesic, it is the geodesic in the free homotopy class of a generator of the fundamental group of $S^1\times \R^2$. It is clear that the length of the core geodesic tends to zero as $p+{\bf i}q\to \infty$. In general, the transformation $A(z)=e^{\ell}z$ for $\ell\neq 0$ gives rise to a geodesic loop covered by $\{0\}\times (0,\infty)$. The {\it complex length} of this loop is $\ell$.

Given a complete, finite volume, hyperbolic manifold with one cusp, $M$, there is a proper embedding of $T^2\times[0,\infty)$ that induces an injection on the level of fundamental groups. In the case of a hyperbolic knot complement, it is natural to take this embedding so that the elements $\alpha$ and $\beta$ of the fundamental group of $T^2$ get mapped to the meridian and longitude respectively. The complete hyperbolic structure on $M$ induces one on $T^2\times\R$ obtained as the corresponding cover. On the level of representations, the representation $\rho_M:\pi_1(M)\to\text{PSL}_2\C$ arising from the hyperbolic structure induces aa representation $\rho_T:\pi_1(T^2)\to \text{PSL}_2\C$ by restriction. Denote that space of all representations, i.e. homomorphisms mod conjugacy of $\pi_1(X)$ into
$\text{PSL}_2\C$ by $\mathcal{R}(X)$.

The heart of the proof of Thurston's Dehn filling theorem is the fact that the image of an open set about $\rho_M\in \mathcal{R}(M)$ under restriction includes an open set about the restriction $\rho_T\in\mathcal{R}(T^2)$, and in particular contains the image of an open set about $\infty\in\overline{\C}$ in $\mathcal{R}(T^2)$. We record this as a proposition here.

\begin{proposition}\label{heart}
If $M$ is a complete hyperbolic manifold with one cusp end $T^2\times(0,\infty)\hookrightarrow M$, $\rho_M\in\mathcal{R}(M)$ is the corresponding representation, and $\rho_T\in\mathcal{R}(T^2)$ is the induced representation, then the image of any neighborhood of $\rho_M$ in $\mathcal{R}(M)$  under the induced map includes an open neighborhood of $\rho_T$ in $\mathcal{R}(T^2)$.
\end{proposition}
Thurston's proof is to remove a suitable, properly embedded line so that the end of the manifold becomes a collar on a surface of genus $2$, make a dimension count, and then analyze the relation induced by glueing back the line, \cite{thurston-notes}. An alternate approach is to linearize the problem about $\rho_M$, using what is effectively the submersion theorem. This is what Cooper, Hodgson, and Kerckhoff do, \cite{chk}.
We will use this result, together with some facts about convergence of sequences of hyperbolic manifolds, that may be found in \cite{ben}.

\begin{theorem}[Exceptional Symmetry Theorem]
Let $M$ be a complete hyperbolic $3$-manifold with one cusp end. Picking a basis for the first homology of this end allows one to identify Dehn fillings with the extended rational numbers. For all but a finite number of $p/q \in \overline{\Q}:=\Q\cup \{\infty\}$, the manifold $M(p/q)$ is hyperbolic and the isometry group $\hbox{\rm Isom}(M(p/q))$ is isomorphic to a subgroup of $\hbox{\rm Isom}(M)$.
\end{theorem}

\begin{proof}

As discussed in Proposition \ref{heart}, the basis of the homology of the end induces a map $i^*:\mathcal{R}(M)\to\mathcal{R}(T^2)$, with image containing an open neighborhood, $\mathcal{U}=\mathcal{U}(\rho_T)$. The inverse image of the neighborhood in the Dehn filling space, $(DF)^{-1}(\mathcal{U})$ is then an open neighborhood of $\infty\in\overline{\C}$, so it contains all but a finite number of
the members of
\[
S:=\{p+q{\bf i}\in \Z[{\bf i}] | \hbox{gcd}(p,q)=1\}\subseteq\overline{\C}\,.
\]
For each point in $(DF)^{-1}(\mathcal{U})\cap S$ there is a corresponding representation in $\mathcal{R}(T^2)$, and a representation $\rho_{p+{\bf i}q}^M$ in $\mathcal{R}(M)$ that maps to it. This representation is the holonomy of possibly several incomplete hyperbolic structures on $M$, but provided one chooses a small enough neighborhood $\mathcal{U}$ there will be a unique hyperbolic structure with this holonomy close to the complete structure on $M$. In the discussion before this theorem, we saw that the completion of this structure on the cover corresponding to the subgroup the end was just the result of performing $p/q$ surgery on the end. In fact, the completion of this incomplete structure on $M$ is just a hyperbolic structure on $M(p/q)$. Up to this point we have just been copying the proof of Thurston's Dehn surgery theorem.

Now the hypothesis that $M$ has one cusp end and is thus diffeomorphic to the interior of a compact manifold with boundary a union of tori implies that the complete hyperbolic metric on $M$ has finite volume (\cite{ben} page 157). The space of complete, finite volume hyperbolic $3$-manifolds, $\mathcal{F}$, has a natural topology arising from the corresponding subgroups of $\text{PSL}_2\C$.  The manifolds in a neighborhood of a fixed hyperbolic $3$-manifold will be similar to it in a number of ways. In particular there is an $\epsilon>0$ and neighborhood (in $\mathcal{F}$) of our starting manifold, $M$, such that the $\epsilon$-thick part of every manifold in this neighborhood is homeomorphic to the $\epsilon$-thick part of $M$. This follows from the proof of Theorem E.2.4 of \cite{ben}.

Recall the thick-thin decomposition of hyperbolic $3$-manifolds. For $\epsilon>0$ the thin part of a hyperbolic manifold, $N$, is the closure of the set of points having radius of injectivity less than $\epsilon$. It is denoted by $N_{(0,\epsilon]}$. The thick part of $N$ is the union of the set of points having injectivity radius greater than or equal to $\epsilon$. It is denoted by $N_{[\epsilon,\infty)}$. The Margulis lemma, \cite{grom}, implies that there is a constant independent of the $3$-manifold such that the $\epsilon$-thin part of a finite volume, orientable $3$-manifold will be a union of tubes $D^2\times S^1$ about short closed geodesics and cusps $T^2\times [0,\infty)$, provided $\epsilon$ is smaller than this Margulis constant, \cite{ben}.
Since our starting manifold, $M$, has just one cusp end, the thin part will consist of the one cusp plus a finite collection of tubes. Taking the $\epsilon$ sufficiently small we can be sure that the thin part of this manifold has no tubes. Thus $M$ is homeomorphic to the interior of its thick part, and the symmetries of the thick part are the same as the symmetries of $M$.

At the start of the proof we saw that there were elements of the representation space corresponding to $p+{\bf i}q$ for $p^2+q^2$ sufficiently large, and these elements
$\rho_{p+{\bf i}q}$ approach $\rho_M$. Furthermore the representations $\rho_{p+{\bf i}q}$ correspond to the surgery manifolds $M(p/q)$ and we have $M(p/q)\to M$ as ${p+{\bf i}q}\to\infty$. Thus when $p^2+q^2$ is sufficiently large the thick parts are homeomorphic, i.e. $M(p/q)_{[\epsilon,\infty)}$ is homeomorphic to $M_{[\epsilon,\infty)}$. Because the thick part is defined geometrically, any isometry $f$ of $M(p/q)$ must take the thick part to itself.

Thus we may associate to each isometry of $M(p/q)$ (for large $p^2+q^2$) the diffeomorphism (and hence mapping class) of $M$ obtained by restricting $f$ to the interior of the thick part of $M(p/q)$. One might worry that this could send a non-trivial isometry to a trivial mapping class, but that would be worrying too much. Indeed,
the fundamental group of $M(p/q)$ is obtained by adding one relation to the fundamental group of $M$ which we now identify with the interior of the thick part. One can take (representatives of) generators of the fundamental group of $M$
and see what happens to each under the image of $f$. Assuming that $f|_M$ is isotopic to the identity, one can let $f|_{M,t}$ represent the isotopy. This gives natural tails for each map in the family of the isotopy. Indeed, when $x_0$ is the base point $\gamma_t(s)=f|_{M,st}(x_0)$ is the natural tail. The automorphism induced by $f$ takes a generator $\alpha$ to $\gamma_1*(f\circ \alpha)*\gamma_1^{-1}$. This is homotopic (rel $x_0$) to $\alpha$ via $\gamma_t*(f|_{M,t}\circ \alpha)*\gamma_t^{-1}$, so the induced automorphism is trivial implying that $f$ itself was trivial to begin with.
\end{proof}

In the process of proving that any finite group may be realized as the isometry group of some hyperbolic $3$-manifold, Kojima \cite{K} mentions this result writing, ``Take any isometry of $M(L,P/Q)$, then since [the link of core geodesics] $L^*$ consists of shortest geodesics by the choice of $P/Q$, [made to insure that these were shortest geodesics] it leaves $L^*$ invariant. Thus $\text{Isom}(M(L;P/Q))=I_{L^*}$.'' Here $I_{L^*}$ denotes the isometries of $M$ fixing $L$. Kojima works in the case of a link consisting of simple geodesics in $M$ with isometries of $M$ acting simply transitively on the components of the link. His $P/Q$ surgery coefficient indicates that he is using the same surgery slope (as determined by the isometries) when filling each component.  Our description above provides a bit more exposition on this quote from his paper.

\begin{figure}[!ht]  
\hskip100bp
\includegraphics[width=50mm]{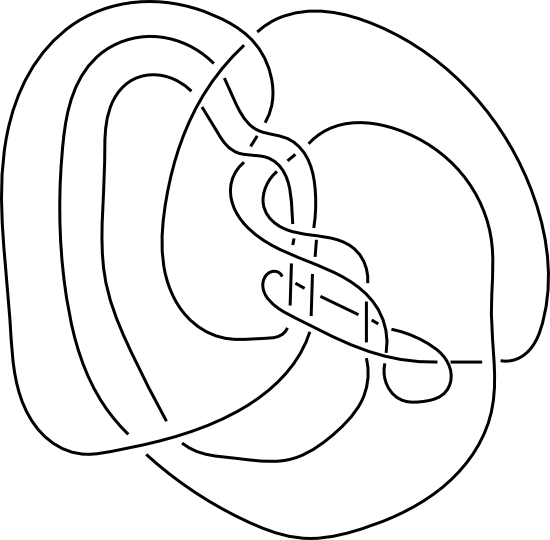}
\caption{An asymmetric knot with an exceptional symmetry}\label{92351111}
\end{figure}

The knot in figure \ref{92351111} demonstrates the exceptional symmetries that can occur as described in the theorem. It has no symmetries, so all but a finite number of Dehn fillings of this knot have no symmetries. This implies that none of these infinitely many manifolds $2$-fold branched cover any $3$-manifold. However, $-2$-surgery on this large knot is equivalent to $2$-surgery on the $8_6$ knot. This knot has the dihedral group of order $4$ as a symmetry group, and these symmetries extend to the surgeries on the knot. Two of the involutions fix points in the knot so the corresponding quotients are $S^3$. The axis of the other involution avoids the knot. Even so, the knot projects to an unknot in the quotient. The induced framing is $1$ so the quotient is still $S^3$. To see that $-2$ surgery on the large knot is the same as $2$ surgery on the $8_6$ knot, one starts with the $9^2_{35}$ link displayed in figure \ref{92n86}. Each component of this link is unknotted. Blowing down (the opposite of a Rolfsen twist) the $1$-framed component leads to the large knot from figure \ref{92351111}. Blowing down the $-1$-framed component leads to the $8_6$ knot.

\begin{figure}[!ht]  
\hskip55bp
\[\begin{array}{cc}\begin{array}{c}
\includegraphics[width=40mm]{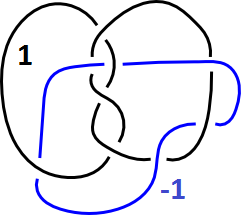} \end{array}&\hskip30bp
\begin{array}{c}\includegraphics[width=45mm]{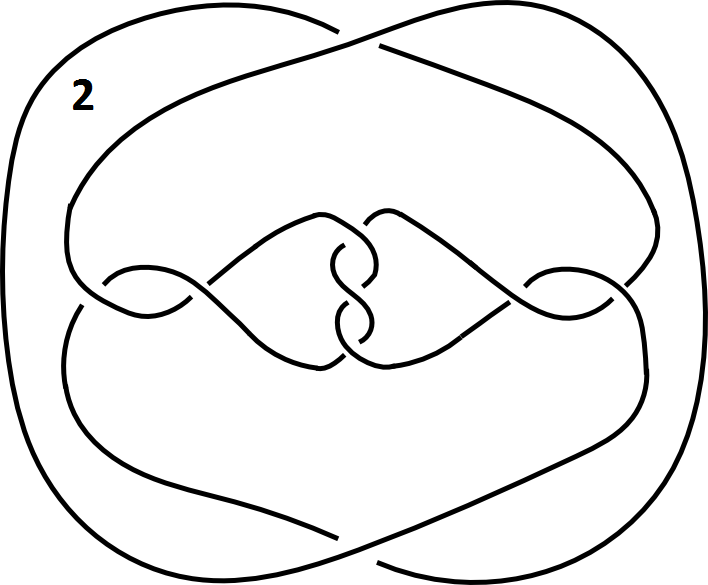}\end{array}
\end{array}
\]
\caption{The $9^2_{35}$ link and $8_{6}$ knot}\label{92n86}
\end{figure}

The theorem states that every symmetry of a generic filling is a symmetry of the open manifold. It does not state that every symmetry of the open manifold extends to the filling. We have seen this in the example of the $8_{17}$ knot displayed in figure \ref{810817}. In fact, for homological reasons a symmetry of type $(S^0,S^0)$ can only extend across trivial or zero surgery. On the other hand we have seen that every symmetry of type $(S^1,S^0)$, or $(S^1,\emptyset)$ does extend across every surgery, so the symmetries extend for the examples that we care most about.

One of the reasons the $5_2$ knot from our first example is a good knot to consider is that it also admits exceptional symmetries, as well as exceptional surgeries. It has symmetry group $\mathbb{Z}_2\oplus\mathbb{Z}_2$ with two involutions of type $(S^1,S^0)$ and one involution of type $(S^1,\emptyset)$. By our earlier discussion these symmetries extend across all surgeries on this knot, and for all but a finite number of exceptions these surgered manifolds have symmetry $\mathbb{Z}_2\oplus\mathbb{Z}_2$. The manifold $S^3_{5_{2}}(1/3)$ is hyperbolic and has symmetry $\mathbb{Z}_2\oplus\mathbb{Z}_2$ as expected.

The manifold $S^3_{5_{2}}(1)$ is not hyperbolic. To see this we do the same trick that we did to understand the exceptional symmetry in the large knot -- we find a two component link half way in between. A Rolfsen twist on an unknot that links the middle $3/2$ twisted band in figure \ref{5_28_5} will untwist it two times. The framing on this new curve will be $1/2$. This operation will unknot the $5_2$ knot. Since the framing on the new unknot is $1$, we can blow it down to see that this manifold is the same as $1/2$ surgery on the right hand trefoil. This in turn is the Seifert Fiber Manifold over $S^2$ with three singular fibers and invariants $\{1, (Oo,0), (-2,1), (-3,1), (-11,2)\}$.

The manifold
$S^3_{5_{2}}(1/2)$ has symmetry $D_{2\cdot 4}$ not $\mathbb{Z}_2\oplus\mathbb{Z}_2$ as expected. To see this, note that a Rolfsen twist about an unknot that links one of the half-twisted bands and reverses the crossing will unknot the $5_2$ knot resulting in a link of two unknotted and algebraically unlinked components. One will have filling $1/2$ and the other will have filling $1$. Doing the Rolfsen twist to untwist the $1/2$-framed component will result in $1$ filling on the $8_3$ knot. The $D_{2\cdot 4}$ is then manifest.


\section{Branched Virtual Fibration}
The virtual fibration theorem is a major new result in the theory of $3$-manifolds.
It states that any hyperbolic $3$-manifold admits a finite cover that fibers over the circle. Thurston asked it as a question in 1982. In 2007 Agol proved that any  manifold with what is known as a virtually residually finite rationally solvable fundamental group is virtually fibered.
Agol and Wise proved the virtual fibering conjecture for any finite volume complete hyperbolic 3-manifold (Agol for closed, Wise for non-compact, \cite{A,W}.)

If finite cover is relaxed to finite {\it branched} cover one can see that any $3$-manifold has a $2$-fold branched cover that fibers over the circle. This branched virtual fibration result was first proved by Sakuma using Heegaard splittings, \cite{sa}. Later Brooks showed that the fibered manifold could always be chosen to be hyperbolic, \cite{Br}, and Montesinos gave an alternate argument for the stronger result using open book decompositions, \cite{mont2}.

\section{Quotient Tabulation}

The symmetries of knots with fewer than $11$ crossings were tabulated in \cite{HeW,KS}. For many knots this is all the information that is needed to understand the possible quotients by $2$-fold branched covers. However, it is not always possible to infer the type of the symmetry of an element of one of these groups from the tabulated information. Furthermore, for symmetries of type $(S^1,\emptyset)$ one needs to know if the image of the knot in the quotient is knotted or unknotted in order to decide if the homology sphere fillings $2$-fold branched cover a non-trivial manifold. Rather than just including the knots that need new information (knotted or unknotted quotient) and referring the reader to the earlier tabulations for the rest, we include all knots in this tabulation. We do just concentrate on the involutions, but remark here that the following knots have higher order symmetries: $4_1$, $6_3$, $7_4$, $7_7$, $8_3$, $8_9$, $8_{12}$, $9_{10}$, $9_{17}$, $9_{23}$, $9_{31}$, $10_{17}$, $10_{33}$, $10_{37}$, $10_{43}$, $10_{45}$, $10_{157}$ (all with $D_4$ symmetry), $9_{41}$, $9_{47}$, $9_{49}$ (all with $D_3$ symmetry), $9_{35}$, $9_{40}$, $9_{48}$, $10_{75}$ (all with $D_6$ symmetry),  $8_{18}$ ($D_8$ symmetry), and $10_{123}$ ($D_{10}$ symmetry).

\renewcommand{\arraystretch}{2}
\hskip-25bp\begin{tabular}{p{8 cm }p{5 cm}}  
\hline
{\bf Symmetry Type} & {\bf Knot List} \\ \hline
No symmetry -- generic surgeries on these knots do not $2$-fold branched cover any manifold. (29 knots) & $9_{32}$, $9_{33}$, $10_{80}$, $10_{82}$, $10_{83}$, $10_{84}$, $10_{85}$, $10_{86}$, $10_{87}$, $10_{90}$, $10_{91}$, $10_{92}$, $10_{93}$, $10_{94}$, $10_{95}$, $10_{102}$, $10_{106}$, $10_{107}$, $10_{110}$, $10_{117}$, $10_{119}$, $10_{148}$, $10_{149}$, $10_{150}$, $10_{151}$, $10_{153}$  \\ \hline

Only type $(S^0,S^0)$ symmetry -- generic surgeries on these knots do not $2$-fold branched cover any manifold, but $0$ surgery will  be a $2$-fold cover of a non-orientable manifold. (7 knots) & $8_{17}$, $10_{79}$, $10_{81}$, $10_{88}$, $10_{109}$, $10_{115}$, $10_{118}$  \\ \hline

Only type $(S^1,S^0)$ symmetry -- generic surgeries on these knots will be $2$-fold branched covers over $S^3$, but do not $2$-fold branched cover any other manifold. (79 knots) & $8_{10}$, $8_{16}$, $8_{20}$, $9_{22}$, $9_{24}$, $9_{25}$, $9_{29}$, $9_{30}$, $9_{34}$, $9_{36}$, $9_{38}$, $9_{39}$, $9_{41}$, $9_{42}$, $9_{43}$, $9_{44}$, $9_{45}$, $9_{47}$, $9_{49}$,
$10_{46}$, $10_{47}$, $10_{48}$, $10_{49}$, $10_{50}$, $10_{51}$, $10_{52}$, $10_{53}$, $10_{54}$, $10_{55}$, $10_{56}$, $10_{57}$, $10_{59}$, $10_{62}$, $10_{65}$, $10_{70}$, $10_{71}$, $10_{72}$, $10_{73}$, $10_{77}$, $10_{89}$, $10_{96}$, $10_{97}$, $10_{100}$, $10_{101}$, $10_{103}$, $10_{104}$, $10_{105}$, $10_{108}$, $10_{111}$, $10_{112}$, $10_{113}$, $10_{114}$, $10_{116}$, $10_{121}$, $10_{125}$, $10_{126}$, $10_{127}$, $10_{128}$, $10_{129}$, $10_{130}$, $10_{131}$, $10_{132}$, $10_{133}$, $10_{134}$, $10_{135}$, $10_{137}$, $10_{140}$, $10_{143}$, $10_{152}$, $10_{154}$, $10_{156}$, $10_{158}$, $10_{159}$, $10_{160}$, $10_{161}$, $10_{162}$, $10_{163}$, $10_{164}$, $10_{165}$ \\ \hline

\end{tabular}

\vfill\newpage
\renewcommand{\arraystretch}{2}
\hskip-25bp\begin{tabular}{p{8 cm }p{5 cm}}  
\hline
{\bf Symmetry Type} & {\bf Knot List} \\ \hline

Only type $(S^1,\emptyset)$ symmetry with unknotted quotient -- generic surgeries on these knots will be $2$-fold branched covers over a lens space. This unique possible quotient will be $S^3$ exactly when the filling yields an integral homology sphere upstairs, i.e. $1/n$ filling.  & $10_{67}$, $10_{147}$  \\ \hline

Only type $(S^1,\emptyset)$ symmetry with knotted quotient -- generic surgeries on these knots will be $2$-fold branched covers over some non-simply-connected manifold, but do not $2$-fold branched cover $S^3$. & $10_{98}$  \\ \hline

Both types $(S^1,S^0)$ symmetry and $(S^1,\emptyset)$ symmetry with unknotted quotient -- generic surgeries on these knots will be $2$-fold branched covers over $S^3$, as well as some lens space (unless the filling yields an integer homology sphere in which case the only possible quotient will be $S^3$). (93 knots) & $4_{1}$, $5_{2}$, $6_{1}$, $6_{2}$, $6_{3}$, $7_{2}$, $7_{3}$, $7_{4}$, $7_{5}$, $7_{6}$, $7_{7}$, $8_{1}$,  $8_{2}$, $8_{3}$, $8_{4}$, $8_{6}$, $8_{7}$,

$9_{2}$, $9_{3}$, $9_{4}$, $9_{5}$, $9_{6}$, $9_{7}$, $9_{8}$, $9_{9}$, $9_{10}$, $9_{11}$, $9_{12}$, $9_{13}$, $9_{14}$, $9_{15}$, $9_{17}$, $9_{18}$, $9_{19}$, $9_{20}$, $9_{21}$, $9_{23}$, $9_{26}$, $9_{27}$, $9_{31}$, $9_{35}$, $9_{37}$,
$9_{46}$, $9_{48}$,

$10_{1}$, $10_{2}$, $10_{3}$, $10_{4}$, $10_{5}$, $10_{6}$, $10_{7}$, $10_{8}$, $10_{9}$, $10_{10}$, $10_{11}$, $10_{12}$, $10_{13}$, $10_{14}$, $10_{15}$, $10_{16}$, $10_{17}$, $10_{18}$, $10_{19}$, $10_{20}$, $10_{21}$, $10_{22}$, $10_{23}$, $10_{24}$, $10_{25}$, $10_{26}$, $10_{27}$, $10_{28}$, $10_{29}$, $10_{30}$, $10_{31}$, $10_{32}$, $10_{33}$, $10_{34}$, $10_{35}$, $10_{36}$, $10_{37}$, $10_{38}$, $10_{39}$, $10_{40}$, $10_{41}$, $10_{42}$, $10_{43}$, $10_{44}$, $10_{45}$, $10_{68}$, $10_{69}$, $10_{74}$, $10_{75}$, $10_{145}$, $10_{146}$
 \\ \hline

Both types $(S^1,S^0)$ symmetry and $(S^0,\emptyset)$ symmetry -- generic surgeries on these knots will only be $2$-fold branched covers over $S^3$. & $10_{99}$, $10_{123}$ \\ \hline

\end{tabular}
\vfill\newpage
\renewcommand{\arraystretch}{2}
\hskip-25bp\begin{tabular}{p{8 cm }p{5 cm}}  
\hline
{\bf Symmetry Type} & {\bf Knot List} \\ \hline

Both types $(S^1,S^0)$ symmetry and $(\emptyset,\emptyset)$ symmetry -- any surgery on one of these knots will be a $2$-fold branched covers over $S^3$, as well as a non-trivial $3$-manifold. & $10_{155}$, $10_{157}$
\\ \hline

Both types $(S^1,S^0)$ symmetry and $(S^1,\emptyset)$ symmetry with knotted quotient -- non-trivial surgeries on these knots will be $2$-fold branched covers over $S^3$, as well as some non-trivial $3$-manifold. (22 knots) &
$8_5$, $8_{15}$, $8_{21}$, $9_{16}$,  $9_{28}$,  $9_{40}$,  $10_{58}$, $10_{60}$, $10_{61}$, $10_{63}$, $10_{66}$, $10_{76}$, $10_{78}$, $10_{120}$, $10_{122}$, $10_{136}$, $10_{138}$, $10_{139}$, $10_{141}$, $10_{142}$, $10_{144}$

\\ \hline

Torus knots with both types $(S^1,S^0)$ symmetry and $(S^1,\emptyset)$ symmetry with unknotted quotient -- any surgery on one of these knots will be $2$-fold branched covers over $S^3$, as well as some lens space (unless the filling yields an integer homology sphere in which case the only possible quotient will be $S^3$. & $3_{1}$, $5_{1}$, $7_{1}$, $9_{1}$  \\ \hline

Torus knots with both types $(S^1,S^0)$ symmetry and $(\emptyset,\emptyset)$ symmetry -- any surgery on one will be a $2$-fold branched cover over $S^3$, as well as a $2$-fold cover (or branched cover) over some other manifold.   & $10_{124}$    \\ \hline

\end{tabular}

\bibliography{2fold}

\end{document}